\documentclass{amsart}
\usepackage[T1]{fontenc}

\usepackage[backend=bibtex]{biblatex} %
\newtheorem{theorem}[subsection]{Theorem}

\newtheorem{corollary}[subsection]{Corollary}

\newtheorem{remark}[subsection]{Remark}

\newtheorem{example}[subsection]{Example}

\newtheorem{proposition}[subsection]{Proposition}

\newtheorem{question}[subsection]{Question}

\newtheorem{problem}[subsection]{Problem}

\title{The Macías topology on integral domains}

\author{Jhixon Macías}

\address{Jhixon Macías\newline
University of Puerto Rico at Mayaguez, Mayaguez, PR, USA\newline
United States of America}
\email{jhixon.macias@upr.edu}

\keywords{Primes, Topology, Integral Domains, Golomb’s topology}

\subjclass[2020]{54A05; 54G05; 54H11}
\begin{document}

\begin{abstract}
% En el presente manuscrito se generaliza, sobre dominios integrales, una reciente topología sobre las enteros positivos generada por la colección de enteros positivos que son primos relativos. Algunas de sus propiedades topológicas son estudiadas. También se estudian algunas  propiedades de esta topología sobre dominios de ideales principales infinitos que no son campo y se obtiene una nueva prueba topológica, diferente a las presentadas en el estilo de H. Furstenberg, de la infinitud de elementos primos (se asume que el conjunto de unidades es finito). Finalmente, algunos problemas son propuestos.

In this manuscript a recent topology on the positive integers generated by the collection of $\{\sigma_n:n\in\mathbb{N}\}$ where $\sigma_n:=\{m: \gcd(n,m)=1\}$ is generalized over integral domains. Some of its topological properties are studied. Properties of this topology on infinite principal ideal domains that are not fields are also explored, and a new topological proof of the infinitude of prime elements is obtained (assuming the set of units is finite or not open), different from those presented in the style of H. Furstenberg. Finally, some problems are proposed.
\end{abstract}

\maketitle
\section{Introduction}
In 1955, H. Furstenberg introduced a topology $\tau_F$ on the integers $\mathbb{Z}$ generated by the collection of arithmetic progressions of the form $a+b\mathbb{Z}$ ($a,b\in \mathbb{Z}, b\geq 1$) and with it presented the first topological proof of the infinitude of prime numbers (now very famous), see \cite{furstenberg1955infinitude}. A couple of years earlier (1953), M. Brown had introduced a topology $\tau_G$ on the natural numbers $\mathbb{N}$ that turns out to be coarser than Furstenberg's topology induced on $\mathbb{N}$. This latter topology is generated by the collection of arithmetic progressions $a+b\mathbb{N}$ with $a\in \mathbb{N}$ and $b\in\mathbb{N}\cup\{0\}$ such that $\gcd(a,b)=1$. It was not until 1959 that the topology introduced by Brown was popularized by S. Golomb (now known as Golomb's topology) who in \cite{golomb1959connected} proves that Dirichlet's theorem on arithmetic progressions is equivalent to the set of prime numbers $\mathbb{P}$ being dense in the topological space $(\mathbb{N},\tau_G)$. Golomb also presents a topological proof of the infinitude of prime numbers, verifying that Furstenberg's argument also works on $(\mathbb{N},\tau_G)$. In 1969, A.M. Kirch \cite{kirch1969countable} defines a topology $\tau_K$ coarser than Golomb's topology generated by the collection of arithmetic progressions $a+b\mathbb{N}$ with $a\in \mathbb{N}, b\in\mathbb{N}\cup\{0\}$ and $a$ square-free such that $\gcd(a,b)=1$. Over the years these topologies have been widely studied, see for example the following works \cite{golomb1962arithmetica, banakh2017continuous, szczuka2010connectedness, kirch1969countable, szyszkowska2022increasing, banakh2019golomb, szczuka2013connections, longhi2023coset, del2022totally, szczuka2013darboux, szyszkowska2022properties, zulfeqarr2019some, lovas2010exotic}. In 1997, J. Knopemacher and S. Porubsky \cite{knopemacher1997topologies} generalize both Furstenberg's and Golomb's topologies in integral domains, and study the arithmetic properties obtained in these equipped with these topologies (especially with the generalization of Golomb's topology). In subsequent years, Golomb spaces have been studied in various algebraic structures, see \cite{clark2019note, spirito2020golomb, spirito2021golomb, spirito2024golomb, orum2016golomb}.

Recently, in \cite{Jhixon2023}, a new topology $\tau_M$ is introduced, called the \textit{Macías topology}, which is generated by the collection of sets $\sigma_n:=\{m\in\mathbb{N}: \gcd(n,m)=1\}$. Additionally, some of its properties are studied. For example, the topological space $(\mathbb{N},\tau_M)$ does not satisfy the $\mathrm{T}_0$ separation axiom, satisfies (vacuously) the $\mathrm{T}_4$ separation axiom, is ultraconnected, hyperconnected (therefore connected, locally connected, path-connected), not countably compact (therefore not compact), but it is limit point compact and pseudocompact. On the other hand, for $n,m\in\mathbb{N}$, it holds that $\textbf{cl}_{(\mathbb{N},\tau_M)}(\{nm\})=\textbf{cl}_{(\mathbb{N},\tau_M)}(\{n\})\cap \textbf{cl}_{(\mathbb{N},\tau_M)}(\{m\})$ where $\textbf{cl}_{(\mathbb{N},\tau_M)}(\{n\})$ denotes the closure of the singleton set $\{n\}$ in the topological space $(\mathbb{N},\tau_M)$. Furthermore, in \cite{jhixon2024}, a new topological proof of the infinitude of prime numbers is presented (different from the proofs of Furstenberg and Golomb). In the same work, the infinitude of any non-empty subset of prime numbers is characterized in the following sense: if $A\subset \mathbb{P}$ (non-empty), then $A$ is infinite if and only if $A$ is dense in $(\mathbb{N},\tau_M)$, see \cite[Theorem 4]{jhixon2024}.

In the present manuscript, the Macías topology is generalized over integral domains that are not fields and some of its basic properties are studied (Section \ref{section2}), and it is compared with Golomb's topology (Subsection \ref{subsection2}). We also study some properties of the Macías topology in principal ideal domains (PIDs) that are not fields (Section \ref{section3}). Furthermore, the \textit{infinitude of prime elements} is characterized and the proof of the infinitude of primes presented in \cite{jhixon2024} is generalized over infinite PIDs with a finite number of units (Subsection \ref{subsection3}).

\section{The Macías topology on integral domains}\label{section2}

Throughout this section, we denote by $R$ a commutative ring with identity $1=1_R\neq 0_R$, and without zero divisors (an integral domain) that is not a field. For each $r\in R$, define
\begin{equation*}
    \sigma_r:=\{s\in R: \langle r\rangle+\langle s \rangle =R\}=\{s\in R \mid \ \ \exists u,v \in R: ur+vs=1 \}
\end{equation*}
where for all $r\in R$, $\langle r \rangle$ is the ideal generated by $r$. Let $\beta_R:=\{\sigma_r: r\in R\}$.

\begin{theorem}
The set $\beta_R$ is a basis for some topology on $R$.
\end{theorem}
\begin{proof}
It is clear that for all $r\in R$, we have $\langle 1 \rangle +\langle r \rangle=R$, so $\sigma_1=R$. Now, observe that $\sigma_{r_1r_2}=\sigma_{r_1}\cap\sigma_{r_2}$ for all $r_1,r_2\in R$. Indeed, let $s\in \sigma_{r_1r_2}$. Then there exist $u,v\in R$ such that
$$1=u(r_1r_2)+vs=(ur_1)r_2+vs=(ur_2)r_1+vs.$$
Thus, $s\in \sigma_{r_1}\cap\sigma_{r_2}$. On the other hand, let $s\in\sigma_{r_1}\cap\sigma_{r_2}$. Then, there exist $u_1,v_1,u_2,v_2\in R$ such that $u_1r_1+v_1s=1$ and $u_2r_2+v_2s=1$. Thus:
\begin{equation*}
    \begin{split}
    1=(u_1r_1+v_1s)(u_2r_2+v_2s) &=u_1r_1u_2r_2+u_1r_1v_2s+v_1su_2r_2+v_1sv_2s\\
    &=(u_1u_2)r_1r_2+(u_1r_1v_2+v_1u_2r_2+v_1sv_2)s,
    \end{split}
\end{equation*}
which implies that $s\in\sigma_{r_1r_2}$. Therefore, $\beta_R$ is a basis for some topology on $R$.
\end{proof}

We denote by $\widetilde{M(R)}$ the topological space $(R,\tau_R)$, where $\tau_R$ is the topology generated by $\beta_R$.

The definition of $\sigma_r$ (as defined) is inspired by the \textit{Bézout's identity}. As much as possible, we want to absorb the properties of \textit{relatively prime} elements over integral domains (even when a $\gcd$ does not necessarily exist for every pair of elements in $R$). Therefore, the study in this manuscript does not focus on $R$, but on $R^0:=R\setminus\{0\}$. In general, for $X\subset R$, $X^0=X\setminus\{0\}$. Thus, we define \textit{the Macías topological space} $M(R)$ as the topological space $(R^0, \tau_{R^0})$ where $\tau_{R^0}$ is the topology generated by the basis $\beta_{R^0}:=\{\sigma_k^0: k\in R^0\}$. The reason why we do not want to consider $\sigma_0$ is that we do not want the units of $R$, denoted by $U(R)$, to be open sets (note that $\sigma_0=U(R)$). This will become evident in Section \ref{subsection3}. For now, we have the following theorem:

\begin{theorem}\label{thsobreUnidades}
The following propositions hold over $R$:
\begin{enumerate}
    \item $u\in U(R)$ if and only if $\sigma_u=R$.
    \item $0\in\sigma_k$ if and only if $k\in U(R)$.
    \item $\textbf{cl}_{\widetilde{M(R)}}(\{0\})\subset \textbf{cl}_{\widetilde{M(R)}}(\{r\})$ for all $r\in R$.
    \item If $u\in U(R)$, then $u\in\sigma_r$ for all $r\in R$.
    \item If $u\in U(R)$, then $\textbf{cl}_{M(R)}(\{u\})=R^0$ and $\textbf{cl}_{\widetilde{M(R)}}(\{u\})=R$.
\end{enumerate}
\end{theorem}
\begin{proof}
    $(1)$, $(2)$, and $(4)$ are immediate. $(3)$ follows from $(2)$ and $(5)$ from $(4)$.
\end{proof}

On the other hand, note that the fundamental properties of $R$ for $\beta_R$ to be a basis are: commutativity and the existence of $1_R$. The condition of having no zero divisors or that $R$ is not a field is initially required due to the general interest in \textit{algebraic-arithmetical domains}, which is generated by previous studies of Golomb's topology on integral domains. We also want to avoid \textit{trivial topologies}. For example, if $R=\mathbb{Z}/2\mathbb{Z}=\{\overline{0},\overline{1}\}$, we have that $M(R)$ is the indiscrete space (see Theorem \ref{thtrivial}).

Now, some basic properties of $\widetilde{M(R)}$ and $M(R)$.

\begin{theorem}
$\widetilde{M(R)}$ is a topological semigroup.
\end{theorem}
\begin{proof}
    Let $\cdot _{R}:R\times R\to R$ be the product of $R$. Note that for each $r\in R$ we have:
    \begin{equation*}
        \begin{split}
            \cdot _{R}^{-1}(\sigma_r)&=\{(x,y)\in R: xy\in \sigma_r\}\\
            &=\{(x,y)\in R: \langle r \rangle+\langle xy\rangle=R\}\\
            &=\{(x,y)\in R: \langle r \rangle+\langle x\rangle=R \ \ \text{and} \ \ \langle r \rangle+\langle y\rangle=R\}\\
            &=\{(x,y)\in R: x,y\in \sigma_r\}\\
            &=\sigma_r\times\sigma_r,\\
        \end{split}
    \end{equation*}
    so $\cdot _{R}$ is continuous. Therefore, $\widetilde{M(R)}$ is a topological semigroup.
\end{proof}

\begin{corollary}
$M(R)$ is a topological semigroup.
\end{corollary}

A topological space $X$ is hyperconnected if there are not two disjoint open sets. Given that $1\in\sigma_r$ for all $r\in R$, we have the following theorem.

\begin{theorem}
$\widetilde{M(R)}$ and $M(R)$ are hyperconnected.
\end{theorem}

Since a hyperconnected space cannot be Hausdorff unless it contains only a single point and every hyperconnected space is both connected and locally connected and moreover extremally disconnected, we have the following corollary.

\begin{corollary}
$\widetilde{M(R)}$ and $M(R)$ are not Hausdorff, but are both connected and locally connected and extremally disconnected.
\end{corollary}

\subsection{The Macías topology vs. the Golomb topology}\label{subsection2}

In \cite{clark2019note}, the topological space $\widetilde{G(R)}$ is defined, which is obtained by equipping $R$ with the topology generated by the collection of \textbf{coprime cosets} ${x+I}$ where $x\in R$ and $I$ is a nonzero ideal of $R$. The Golomb topological space $G(R)$ is obtained by equipping $R^0$ with the subspace topology of $\widetilde{G(R)}$.

In the case of $\mathbb{N}$, the Macías topology is coarser than the Golomb topology. Therefore, it is interesting to ask if this occurs in general over $R$. The answer to this latter question is affirmative.

\begin{theorem}
Every open set in $\widetilde{M(R)}$ (except for $\sigma_0=U(R)$) is open in $\widetilde{G(R)}$.
\end{theorem}
\begin{proof}
Let $\sigma_k\in \beta_R$. If $x\in\sigma_k$, then there exist $u,v\in R$ such that $ux+vr=1$. Let $r\in R$, then $1=u(x+kr)+(v-uk)r$. Thus, $x+rk\in\sigma_k$ for all $r\in R$. On the other hand, consider the element $r+nk$ where $\langle r\rangle + \langle k\rangle=R$ and $n\in R$. Since $r\in\sigma_k$, there exist $u,v\in R$ such that $ur+vk=1$. Then $u(r+nk)+(v-un)k=1$. Thus, $r+nk\in\sigma_k$. Therefore, for all $\sigma_k\in \beta_R$, we have:

\begin{equation*}
   \sigma_k= \bigcup_{\substack{r\in R\\ \langle r \rangle + \langle k \rangle =R}}r+\langle k \rangle .
\end{equation*}
From this last equality, the result follows.
\end{proof}

Note from this last result that every open set in $M(R)$ is open in $G(R)$. Of course, in $M(R)$ and $G(R)$ the units do not form an open set. Thus, we have the following corollary.

\begin{corollary}\label{corcompaacion}
    The Macías topology is strictly coarser than the Golomb topology.
\end{corollary}

Now, in \cite{clark2019note} the following theorem is proven:

\begin{theorem}[Lemma 1]
For an integral domain $R$, the following are equivalent
\begin{enumerate}
    \item The space $\widetilde{G(R)}$ is indiscrete.
    \item The space $G(R)$ is indiscrete.
    \item The ring $R$ is a field.
\end{enumerate}
\end{theorem}

Therefore, by virtue of Corollary \ref{corcompaacion}, we have the following theorem:

\begin{theorem}\label{thtrivial}
For an integral domain $R$, the following are equivalent
\begin{enumerate}
    \item The space $M(R)$ is indiscrete.
    \item The ring $R$ is a field.
\end{enumerate}
\end{theorem}

Note that $M(R)$ can be indiscrete, but $\widetilde{M(R)}$ (for $R$ an integral domain) cannot. Consider the following example.

\begin{example}
Consider the Boolean Field $R=\mathbb{Z}/2\mathbb{Z}=\{\overline{0},\overline{1}\}$, then $\widetilde{M(R)}$ is not the indiscrete space. Indeed, $\beta_R=\{\{1\},R\}$.
\end{example}

Finally, recall that if $R$ is a domain that is not a field, then $R$ is infinite and every nonzero ideal $I$ of $R$ is infinite. Now, since $R$ is an integral domain, the map $R \to I$ defined by $r \mapsto rx$ ($x\in I^0$) is an injection. Therefore, the non-empty open sets in $G(R)$ are infinite, and hence, the non-empty open sets in $M(R)$ are infinite. Another way to see that the non-empty open sets in $M(R)$ are infinite is the following: since $1 \in \sigma_k$ for all $k \in R$, then $1+rk \in \sigma_k$ for all $r,k \in R$. Since $R$ is an integral domain, if $r_1 \neq r_2$, then $1+r_1k \neq 1+r_2k$. Therefore, $\sigma_k$ is infinite for all $k \in R$. This last property is crucial for obtaining a \textit{topological proof of the infinitude of primes} over PIDs.

%%%%%

\section{The Macías topology over principal ideal domains}\label{section3}
From this point on, unless otherwise indicated or an additional condition is assumed, we will suppose that $R$ is an infinite PID (not a field).

\begin{theorem}
    Let $p\in R$ be irreducible (prime). Then $\textbf{cl}_{\widetilde{M(R)}}(\{p\})= \langle p\rangle$.
\end{theorem}
\begin{proof}
Let $x\in \textbf{cl}_{\widetilde{M(R)}}(\{p\})$. Suppose $x\notin \langle p \rangle$. Since $R$ is a PID, $\langle p \rangle$ is maximal in $R$. Thus, $R/\langle p \rangle$ is a field, and hence the coset $x+\langle p \rangle$ is invertible. Thus, there exist $u,v\in R$ such that $ux+vp=1$, implying $x\in\sigma_p$. This implies $p\in\sigma_p$, which is absurd since $p$ should be a unit in this case. On the other hand, let $x\in \langle p\rangle$. Then $x=rp$ for some $r\in R$. Let $\sigma_k\in\beta_R$ such that $x\in \sigma_k$. Thus, there exist $u,v\in R$ such that $1=uk+v(rp)=uk+(vr)p$, implying $p\in\sigma_k$. Therefore, $x\in  \textbf{cl}_{M(R)}(\{p\})$. 
\end{proof}
\begin{corollary}
     Let $p\in R$ be irreducible (prime). Then $\textbf{cl}_{M(R)}(\{p\})= \langle p\rangle^0$.
\end{corollary}

\begin{corollary}\label{corclausura1}
    Let $r\in R^0-U(R)$. We can write $r =p_1{}^{m_1}p_2{}^{m_2}...p_n{}^{m_n}$  where $p_i$ is irreducible (prime) such that $\langle p_i\rangle+\langle p_j\rangle = R$ for $i\neq j$. Then 
    $$\textbf{cl}_{\widetilde{M(R)}}(\{r\})=\bigcap_{i}\langle p_i \rangle$$
\end{corollary}

\begin{proof}
Let $x\in \bigcap_{i}\langle p_i \rangle$. Then, by the previous theorem, $x\in \textbf{cl}_{\widetilde{M(R)}}(\{p_i\})$ for every $i$. Thus, for any basic element $\sigma_k\in \beta_R$ such that $x\in\sigma_k$, $p_i\in\sigma_k$ for every $i$. Hence, $r=p_1{}^{m_1}p_2{}^{m_2}...p_n{}^{m_n}\in\sigma_k$. Therefore, $x\in \textbf{cl}_{\widetilde{M(R)}}(\{r\})$.  On the other hand, let $x\in \textbf{cl}_{\widetilde{M(R)}}(\{r\})$. Then, for every basic element $\sigma_k\in\beta_{R}$ such that $x\in\sigma_k$, $r=p_1{}^{m_1}p_2{}^{m_2}...p_n{}^{m_n}\in\sigma_k$, from which it follows that $p_i\in\sigma_k$ for every $i$. Thus, $x\in\textbf{cl}_{\widetilde{M(R)}}(\{p_i\})$ for every $i$, and by the previous theorem, $x\in \langle p_i \rangle$ for every $i$. Therefore, $\textbf{cl}_{\widetilde{M(R)}}(\{r\})=\bigcap_{i}\langle p_i\rangle$.
\end{proof}

\begin{corollary}\label{corclausura2}
    Let $r\in R^0-U(R)$. We can write $r =p_1{}^{m_1}p_2{}^{m_2}...p_n{}^{m_n}$  where $p_i$ is irreducible (prime) such that $\langle p_i\rangle+\langle p_j\rangle = R$ for $i\neq j$. Then 
    $$\textbf{cl}_{M(R)}(\{r\})=\bigcap_{i}\langle p_i \rangle ^0$$
\end{corollary}

\begin{remark}
Note that if $x,y\in\sigma_k$, then $xy\in\sigma_k$. Of course, if $x,y\in\sigma_k$, there exist $u_x,v_x,u_y,v_y\in R$ such that $1=u_xk+v_xx=u_yk+v_yy$. Thus, $(u_xku_y+u_xv_yy+v_xxu_y)k+(v_xv_y)xy=1$, implying $xy\in \sigma_k$.
\end{remark}

\begin{corollary}\label{corclausuraproducto1}
Let $x,y$ be any elements in $R$. Then
    $$\textbf{cl}_{\widetilde{M(R)}}(\{xy\}) = \textbf{cl}_{\widetilde{M(R)}}(\{x\}) \cap \textbf{cl}_{\widetilde{M(R)}}(\{y\}).$$
\end{corollary}
\begin{proof}
If $x$ or $y$ is a unit or $x=0$ or $y=0$, the result is clear, see Theorem \ref{thsobreUnidades}. Thus, suppose neither $x$ nor $y$ are units and neither are 0. Then we can have $x=p_1{}^{a_1}\cdots p_r{}^{a_r}$ and  $y=q_1{}^{b_1}\cdots q_s{}^{b_s} $ written in their prime decompositions. Then, by Corollary \ref{corclausura1},
\begin{equation*}
    \textbf{cl}_{\widetilde{M(R)}}(\{x\})\cap \textbf{cl}_{\widetilde{M(R)}}(\{x\})=\bigcap_i\langle p_i\rangle\cap\bigcap_j\langle q_j\rangle= \textbf{cl}_{\widetilde{M(R)}}(\{xy\}).
\end{equation*}
The proof is complete.
\end{proof}

Similarly, we prove the following corollary.

\begin{corollary}
Let $x,y$ be any elements in $R^0$. Then
    $$\textbf{cl}_{M(R)}(\{xy\}) = \textbf{cl}_{M(R)}(\{x\}) \cap \textbf{cl}_{M(R)}(\{y\}).$$
\end{corollary}

A topological space $X$ is ultraconnected if every two non-empty closed subsets are not disjoint. Equivalently, $X$ is ultraconnected if and only if the closure of two distinct points always have non-trivial intersection.

\begin{corollary}
$\widetilde{M(R)}$ is ultraconnected.
\end{corollary}
\begin{proof}
Let $F$ and $G$ be two non-trivial closed sets  in
$\widetilde{M(R)}$. Let $x \in F$ and  $y \in G$. 
Then, by Corollary \ref{corclausuraproducto1}, we have that
\begin{equation*}
\{x\cdot y\}\subset\textbf{cl}_{\widetilde{M(R)}}(\{xy\}) = \textbf{cl}_{\widetilde{M(R)}}(\{x\}) \cap \textbf{cl}_{\widetilde{M(R)}}(\{y\})
\subset F\cap G.
\end{equation*}
Therefore, $F\cap G\neq\emptyset$.  
\end{proof}

Similarly, we prove the following corollary.

\begin{corollary}
    $M(R)$ is ultraconnected.
\end{corollary}

Since every ultraconnected space is normal, limit point compact, pseudocompact and path-connected, we have the following corollary.

\begin{corollary}
    $\widetilde{M(R)}$ and $M(R)$ are normal, limit point compact, pseudocompact and path-connected.
\end{corollary}

A $\mathrm{T_0}$ space is a topological space in which every pair of distinct points is topologically distinguishable. That is, for any two different points $x$ and $y$ there is an open set that contains one of these points and not the other. 

For each $x\in R^0-U(R)$, define $$\mathcal{F}(x):=\{p\in R^0: p \ \ \text{is an irreducible element and } \ \ p\mid x\}.$$

\begin{theorem}
    Let $x,y\in R^0-U(R)$. Then, $x$ and $y$ are topologically indistinguishable if and only if $\mathcal{F}(x)=\mathcal{F}(y)$.
\end{theorem}
\begin{proof}
The proof is straightforward.
\end{proof}

\begin{corollary}
$M(R)$ does not satisfy the $\mathrm{T}_0$ separation axiom.
\end{corollary}

Up to this point it's worth asking the following questions:

\begin{question}
    Let $R$ be an integral domain that is not a PID. Is $\widetilde{M(R)}$ or $M(R)$ ultraconnected?
\end{question}

\begin{question}
    Let $R$ be an integral domain that is not a PID. Is $\widetilde{M(R)}$ or $M(R)$ not satisfying the $\mathrm{T}_0$ axiom?
\end{question}

% \begin{example}
%     Note que si $R$ no es un dominio integral, se puede tener que $M(R)$ no sea ultraconexo. Por ejemplo, considere $R=\mathbb{Z}/6\mathbb{Z}$. Luego, $\sigma^0_{\overline{2}}=\{\overline{1},\overline{3},\overline{5}\}$ y $\sigma^0_{\overline{3}}=\{\overline{1},\overline{2},\overline{4},\overline{5}\}$. Así $\{\overline{3}\}$ y $\{\overline{2},\overline{5}\}$ son cerrados y disjuntos en $M(R)$ y por tanto este no es ultraconexo. Sin embargo $\widetilde{M(R)}$ si es ultraconexo ($\{\overline{0},\overline{3}\}$ y $\{\overline{0},\overline{2},\overline{5}\}$ son cerrados en $\widetilde{M(R)}$).
% \end{example}

\begin{example}
    Note that if $R$ is not an integral domain, it is possible for $M(R)$ to not be ultraconnected. For example, consider $R=\mathbb{Z}/6\mathbb{Z}$. Then, $\sigma^0_{\overline{2}}=\{\overline{1},\overline{3},\overline{5}\}$ and $\sigma^0_{\overline{3}}=\{\overline{1},\overline{2},\overline{4},\overline{5}\}$. Thus, $\{\overline{3}\}$ and $\{\overline{2},\overline{4}\}$ are closed and disjoint in $M(R)$, hence it is not ultraconnected. However, $\widetilde{M(R)}$ is ultraconnected (since $\{\overline{0},\overline{3}\}$ and $\{\overline{0},\overline{2},\overline{4}\}$ are closed in $\widetilde{M(R)}$).
\end{example}

\subsection{On the infinitude of primes}\label{subsection3} In this subsection, $R$ is an infinite PID (not a field) with a finite number of units. We aim to prove, topologically, the following proposition.

\begin{proposition}\label{proprimos}
   $R$ has infinitely many maximal ideals.
\end{proposition}

The idea is to generalize the topological proof of the infinitude of prime numbers, presented in \cite{jhixon2024}. In this regard, we present the following result.

\begin{theorem}\label{thcarprimes}
 Let $\mathcal{P}$ be the collection of prime elements of $R$. Then, $\mathcal{P}$ is infinite if and only if $\mathcal{P}$ is dense in $M(R)$.
\end{theorem}

\begin{proof}
Suppose $\mathcal{P}$ is infinite. Let $\sigma_k^0\in\beta_{R^0}$. Then, we can take $p\in\mathcal{P}$ such that $\langle k \rangle + \langle p \rangle =R $. Thus $\sigma_k^0\cap \mathcal{P}\neq\emptyset$. Therefore, $\mathcal{P}$ is dense in $M(R)$. Conversely, suppose $\mathcal{P}$ is dense in $M(R)$ and consider a finite collection of elements from $\mathcal{P}$, say $\{p_1,p_2,\dots , p_n\}$. Let $x=p_1\cdot p_2\cdots p_n$. Then, $x\in R^0$. Now, observe that $\sigma_x^0\neq U(R)$, since $\sigma_x^0$ is infinite. Therefore, since $\mathcal{P}$ is dense in $M(R)$, there exists a $q\in\mathcal{P}\setminus\{p_1,p_2,\dots p_n\}$ such that $q\in \sigma_{x}^0\cap\mathcal{P}$. Hence, $\mathcal{P}$ is dense in $M(R)$.
\end{proof}

\begin{theorem}
 If $\mathcal{P}$ is dense in $R^0\setminus U(R)$ with the induced subspace topology from $M(R)$, denoted by $M(R^0\setminus U(R))$, then $\mathcal{P}$ is dense in $M(R)$.
\end{theorem}

\begin{proof}
Note that $R^0\setminus U(R)$ is dense in $M(R)$. The result follows from the transitive property of density.
\end{proof}

\begin{theorem}
 $\mathcal{P}$ is dense in $R^0\setminus U(R)$ with the induced subspace topology from $M(R)$.
\end{theorem}

\begin{proof}
In any topological space, the union of closures of subsets of that space is contained in the closure of the union of those sets. Therefore,
$$\displaystyle\bigcup_{p\in\mathcal{P}}\mathbf{cl}_{M(R^0\setminus U(R))}(\{p\})\subset\mathbf{cl}_{M(R^0\setminus U(R))}(\mathcal{P})\subset R^0\setminus U(R).$$ 

Moreover, by Corollary \ref{corclausura2}, we have

$$\bigcup_{p\in\mathcal{P}}\mathbf{cl}_{M(R^0\setminus U(R))}(\{p\})=\bigcup_{p\in\mathcal{P}}(\mathbf{cl}_{M(R^0)}(\{p\})\cap 
(R^0\setminus U(R)))=\bigcup_{p\in\mathcal{P}}\langle p \rangle ^0.$$

Now, since $R$ is a unique factorization domain, we have $R^0\setminus U(R)=\bigcup_{p\in \mathcal{P}}\langle p \rangle ^0$. Therefore, $\mathbf{cl}_{M(R^0\setminus U(R))}(\mathcal{P})=R^0\setminus U(R)$.
\end{proof}

The last three theorems establish Proposition \ref{proprimos}.  Thus, we obtain a new topological proof of the infinitude of primes for rings such as the integers $\mathbb{Z}$, the Gaussian integers $\mathbb{Z}[i]$, the Eisenstein integers $\mathbb{Z}[\omega]$ (where $\omega$ is a primitive cube root of 1), etc.

Note that the argument to prove Theorem \ref{thcarprimes} works because under the given hypothesis for $U(R)$, it is guaranteed that it is not open. Therefore, we have the following corollary.

\begin{corollary}
If $U(R)$ is not open in $M(R)$, then $R$ has infinitely many maximal ideals.
\end{corollary}

Finally, we want to present the following theorem.

\begin{theorem}
    $M(R)$ is not compact.
\end{theorem}

\begin{proof}
The proof is straightforward. Consider the cover $\{\sigma_p^0:p\in \mathcal{P}\}$.
\end{proof}
 
\section{Final Comment}

This manuscript is an introductory work that generalizes the topological space $X:=(\mathbb{N},\tau_M)$ in the following sense: If $R=\mathbb{Z}$, then $X$ is the topological space obtained by equipping $\mathbb{N}$ with the subspace topology induced by $M(R)$. Furthermore, it can be widely extended (the studies conducted on the Golomb topology can be replicated as much as possible). For example, one can study the properties of the quotient topology induced by the Macías topology on the semigroup of all associate-classes of non-zero elements in an integral domain $R$, similar to what was done with the Golomb topology in \cite{knopemacher1997topologies}. A similar study to that conducted for the Golomb topology in \cite{clark2019note} can also be carried out. Moreover, one can study the homeomorphism problem in $M(R)$, just as has been done and continues to be done with the Golomb topology, see \cite[Section 3]{clark2019note} . In this regard, the following problem is posed.

\begin{problem}
Let $R$ and $S$ be countably infinite semiprimitive integral domains. Decide whether $M(R)$ and $M(S)$ are homeomorphic.
\end{problem}
\printbibliography
\end{document}